\documentclass[12pt,a4paper]{article}
\usepackage{latexsym,amssymb,amsfonts,amsmath,amsthm,nccmath,float,enumitem,setspace,bm,authblk,longtable,cite}
\usepackage[usenames,dvipsnames]{xcolor}
\usepackage[hidelinks]{hyperref}
\usepackage[margin=2cm]{geometry}
\usepackage{tikz}
\usepackage{color}
\usepackage{multirow}
\usepackage{caption}

\setlist{topsep=3pt,partopsep=0pt,itemsep=1pt,parsep=0pt}

\hypersetup{
	colorlinks,
	linkcolor={red!80!black},
	citecolor={blue!80!black},
	urlcolor={blue!80!black}
}

\numberwithin{equation}{section}
\newtheorem{Theorem}{Theorem}[section]
\newtheorem{Proposition}{Proposition}[section]

\newtheorem{Conjecture}{Conjecture}[section]
\newtheorem{Remark}{Remark}[section]
\newtheorem{Lemma}{Lemma}[section]

\newtheorem{Definition}{Definition}[section]
\newtheorem{Claim}{Claim}[section]

\allowdisplaybreaks

\begin{document}

\title{Variations on Bollob\'{a}s systems of $d$-partitions}

\author[a]{Yu Fang}
\author[a]{Xiaomiao Wang}
\author[b]{Tao Feng}
\affil[a]{School of Mathematics and Statistics, Ningbo University, Ningbo 315211, P. R. China}
\affil[b]{School of Mathematics and Statistics, Beijing Jiaotong University, Beijing 100044, P. R. China}
\affil[ ]{fangymath@163.com; wangxiaomiao@nbu.edu.cn; tfeng@bjtu.edu.cn}
\renewcommand*{\Affilfont}{\small\it}
\renewcommand\Authands{ and }
\date{}

\maketitle

\footnotetext{Supported by Zhejiang Provincial Natural Science Foundation of China under Grant LY24A010001 (X. Wang), and by NSFC under Grant 12271023 (T. Feng).}

\begin{abstract}
This paper investigates five kinds of systems of $d$-partitions of $[n]$, including symmetric Bollob\'{a}s systems, strong Bollob\'{a}s systems, Bollob\'{a}s systems, skew Bollob\'{a}s systems, and weak Bollob\'{a}s systems. Many known results on variations of Bollob\'{a}s systems are unified. Especially we give a negative answer to a conjecture on Bollob\'{a}s systems of $d$-partitions of $[n]$ that was presented by Heged\"{u}s and Frankl [European J. Comb., 120 (2024), 103983]. Even though this conjecture does not hold for general Bollob\'{a}s systems, we show that it holds for strong Bollob\'{a}s systems of $d$-partitions of $[n]$.
\end{abstract}

\section{Introduction}

Throughout this paper, for a positive integer $n$, write $[n]=\{1,\dots,n\}$. Motivated by a problem in graph theory, Bollob\'{a}s \cite{B} introduced the concept of set-pair systems in 1965.

\begin{Definition}\label{def-pair}
Let $\mathcal{P}=\{(A_i^{(1)},A_i^{(2)})\}_{1\le i\le m}$, where $A_i^{(1)},A_i^{(2)}\subseteq [n]$ and $A_i^{(1)}\cap A_i^{(2)}=\emptyset$ for all $i\in[m]$. Then

\begin{itemize}
  \item $\mathcal{P}$ is called a \emph{Bollob\'{a}s system} if $A_i^{(1)}\cap A_j^{(2)}\ne\emptyset$ for any $1\le i, j\le m$ and $i\ne j$.
  \item $\mathcal{P}$ is called a \emph{skew Bollob\'{a}s system} if $A_i^{(1)}\cap A_j^{(2)}\ne\emptyset$ for any $1\le i<j\le m$.
\end{itemize}
\end{Definition}

Clearly a Bollob\'{a}s system is also a skew Bollob\'{a}s system. Bollob\'{a}s \cite{B} proved the following inequality using a combinatorial approach.

\begin{Theorem}{\rm \cite{B}} \label{B-pair}
Let $\mathcal{P}=\{(A_i^{(1)},A_i^{(2)})\}_{1\leq i\leq m}$ be a Bollob\'{a}s system. Then
$$\sum_{i=1}^m\binom{|A_i^{(1)}|+|A_i^{(2)}|}{|A_i^{(1)}|}^{-1}\leq 1.$$
\end{Theorem}

Theorem \ref{B-pair} implies immediately the following result.

\begin{Theorem}\label{B-number-pair}
Let $\mathcal{P}=\{(A_i^{(1)},A_i^{(2)})\}_{1\leq i\leq m}$ be a Bollob\'{a}s system with $|A_i^{(1)}|\leq a_1$ and $|A_i^{(2)}|\leq a_2$ for every $i\in[m]$. Then
$$m\leq\binom{a_1+a_2}{a_1}.$$
\end{Theorem}


Theorem \ref{B-number-pair} was independently rediscovered by Katona \cite{katona} and Jaeger and Payan \cite{jp}, thereby bringing Bollob\'{a}s systems back to the forefront of extremal set theory. A striking feature of this theorem is that the upper bound depends only on $a_1$ and $a_2$, and not on the size of the ground set. Lov\'{a}sz \cite{l} used exterior algebra methods to give another proof of Theorem \ref{B-number-pair}. This method gives an elegant argument that naturally extends to subspaces of a finite-dimensional vector space, and a set-pair system version follows immediately.

Surprisingly, the same upper bound as stated in Theorem \ref{B-number-pair} remains valid even for skew Bollob\'{a}s systems.

\begin{Theorem}{\rm \cite{frankl,kalai}}\label{skew-number-pair}
Let $\mathcal{P}=\{(A_i^{(1)},A_i^{(2)})\}_{1\leq i\leq m}$ be a skew Bollob\'{a}s system with $|A_i^{(1)}|\leq a_1$ and $|A_i^{(2)}|\leq a_2$ for every $i\in[m]$. Then
$$m\leq\binom{a_1+a_2}{a_1}.$$
\end{Theorem}

However, the inequality in Theorem \ref{B-pair} cannot be applied to skew Bollob\'{a}s systems. Very recently, Heged\"{u}s and Frankl \cite{hf} established the following result.

\begin{Theorem}{\rm \cite[Theorem 1.4]{hf}}\label{skew-weight-pair}
Let $\mathcal{P}=\{(A_i^{(1)},A_i^{(2)})\}_{1\leq i\leq m}$ be a skew Bollob\'{a}s system on $[n]$. Then
$$\sum_{i=1}^m\binom{|A_i^{(1)}|+|A_i^{(2)}|}{|A_i^{(1)}|}^{-1}\leq n+1,$$
and the upper bound is tight.
\end{Theorem}

Skew Bollob\'{a}s systems find their applications in various fields, including topology (cf. \cite{bj}), extremal graph theory (cf. \cite{ap}) and matrix complexity theory (cf. \cite{t}). For further reading on  Bollob\'{a}s-type theorems, the reader is referred to \cite{f,kkk,nv,sw,talbot,ykxzg}.

Let $d\geq2$ and $A^{(1)},\ldots,A^{(d)}\subseteq [n]$. We call $(A^{(1)},\dots,A^{(d)})$ a \emph{$d$-partition} of $[n]$ if $A^{(1)}\cup\dots\cup A^{(d)}\subseteq [n]$, and $A^{(p)}\cap A^{(q)}=\emptyset$ for $1\le p<q\le d$. Furthermore, the $d$-partition of $[n]$ is said to be \emph{full} if $A^{(1)}\cup\dots\cup A^{(d)}=[n]$. For non-negative integers $a_1,\ldots,a_d$ satisfying $a_1+\cdots+a_d\le n$, the number of all $d$-partitions $A^{(1)}\cup\cdots\cup A^{(d)}\subseteq[n]$ with $|A^{(r)}|=a_r$, $r\in[d]$, is
$$\binom{n}{a_1,\ldots,a_d}:=\frac{n!}{a_1!\cdots a_d!(n-a_1-\cdots-a_d)!}.$$

\begin{Definition}\label{def-d}
Let $d\geq2$. Let $\mathcal{D}=\{(A_i^{(1)},\dots,A_i^{(d)})\}_{1\le i\le m}$ be a family of $d$-partitions of $[n]$. Then

\begin{itemize}
\item[$(1)$] $\mathcal{D}$ is called a \emph{weak Bollob\'{a}s system} if for any $1\le i<j\le m$, there exist $1\le p< q\le d$ such that $A_i^{(p)}\cap A_j^{(q)}\ne\emptyset$ or $A_j^{(p)}\cap A_i^{(q)}\ne\emptyset$;

\item[$(2)$] $\mathcal{D}$ is called a \emph{skew Bollob\'{a}s system} if for any $1\le i<j\le m$, there exist $1\le p< q\le d$ such that $A_i^{(p)}\cap A_j^{(q)}\ne\emptyset$;

\item[$(3)$] $\mathcal{D}$ is called a \emph{Bollob\'{a}s system} if for any $1\le i, j\le m$ and $i\ne j$, there exist $1\le p< q\le d$ such that $A_i^{(p)}\cap A_j^{(q)}\ne\emptyset$;

\item[$(4)$] $\mathcal{D}$ is called a \emph{strong Bollob\'{a}s system} if for any $1\leq i<j\leq m$, there exist $1\leq u_1<u_2\leq d$ and $1\leq v_1<v_2\leq d$ such that $A_i^{(u_1)}\cap A_j^{(v_2)}\neq \emptyset$ and $A_i^{(u_2)}\cap A_j^{(v_1)}\neq \emptyset$;

\item[$(5)$] $\mathcal{{D}}$ is called a \emph{symmetric Bollob\'{a}s system} if for any $1\le i<j\le m$, there exist $1\le p< q\le d$ such that $A_i^{(p)}\cap A_j^{(q)}\ne\emptyset$ and $A_j^{(p)}\cap A_i^{(q)}\ne\emptyset$.
\end{itemize}	
\end{Definition}

For any $d\geq2$, (5)$\Rightarrow$(4)$\Rightarrow$(3)$\Rightarrow$(2)$\Rightarrow$(1). Note that (3) $\not\Rightarrow$(4). For example, $\mathcal{D}=\{(A_i^{(1)},A_i^{(2)},$ $A_i^{(3)})\}_{1\le i\le 2}$ is a Bollob\'{a}s system of $3$-partitions of $[2]$, where
\begin{align*}
(A_1^{(1)},A_1^{(2)},A_1^{(3)})=&(\{1\},\emptyset,\{2\});\\
(A_2^{(1)},A_2^{(2)},A_2^{(3)})=&(\emptyset,\{1,2\},\emptyset).
\end{align*}
However, $\mathcal D$ is not a strong Bollob\'{a}s system.

Particularly when $d=2$, (3)$\Leftrightarrow$(4)$\Leftrightarrow$(5)$\Leftrightarrow$ Bollob\'{a}s systems in Definition \ref{def-pair}, and (2)$\Leftrightarrow$ skew Bollob\'{a}s systems in Definition \ref{def-pair}.

\subsection{Skew Bollob\'{a}s systems}

Heged\"{u}s and Frankl \cite{hf} proposed the concept of skew Bollob\'{a}s systems of $d$-partitions of $[n]$ and established the following result, which is a generalization of Theorem \ref{skew-weight-pair}.

\begin{Theorem} {\rm \cite[Example 2 and Theorem 4.2]{hf}} \label{skew-d-partition}
Let $\mathcal{D}=\{(A_i^{(1)},\ldots,A_i^{(d)})\} _{1\leq i\leq m}$ be a skew Bollob\'{a}s system of $d$-partitions of $[n]$. Then
$$\sum_{i=1}^{m} \binom{\sum_{r=1}^{d}|A_i^{(r)}|}{|A_{i}^{(1)}|,\ldots,|A_{i}^{(d)}|}^{-1}\le \binom{n+d-1}{d-1},$$
and the upper bound is tight.
\end{Theorem}

Alon \cite{alon} generalized Theorem \ref{skew-number-pair}. He examined the maximum size of a skew Bollob\'{a}s set-pair system in which every set has a fixed size in each part of the ground set.

\begin{Theorem}{\rm \cite{alon}}\label{Theorem-Alon}
Suppose that $[n]$ is the disjoint union of some sets $X_1,\ldots,X_e$. Let $\mathcal{P}=\{(A_i^{(1)},A_i^{(2)})\}_{1\leq i\leq m}$ be a skew Bollob\'{a}s system on $[n]$ with $|A_i^{(1)}\cap X_k|=a_{1,k}$ and $|A_i^{(2)}\cap X_k|=a_{2,k}$ for every $i\in[m]$ and $k\in[e]$. Then
$$m\leq\prod_{k=1}^{e}\binom{a_{1,k}+a_{2,k}}{a_{1,k}}.$$
\end{Theorem}

Yue \cite[Theorem 1.7]{yue} provided a nonuniform version of Theorem \ref{Theorem-Alon}. This paper generalizes Yue's result from $d=2$ to any $d$.

\begin{Theorem}\label{X-skew-d-partition}
Suppose that $[n]$ is the disjoint union of some sets $X_1,\ldots,X_e$. Let $\mathcal{D}=\{(A_i^{(1)},\dots, A_i^{(d)})\}_{1\le i\le m}$ be a skew Bollob\'{a}s system of $d$-partitions of $[n]$. Let $S=\bigcup_{i=1}^m\bigcup_{r=1}^d A_i^{(r)}$ and $s_k=|S\cap X_k|$ for $k\in[e]$. Then
$$\sum_{i=1}^{m}\left(\prod_{k=1}^{e}\binom{\sum_{r=1}^{d}|A_i^{(r)}\cap X_k|}{|A_i^{(1)}\cap X_k|,\ldots,|A_{i}^{(d)}\cap X_{k}|}\right)^{-1}\le \prod_{k=1}^{e}\binom{s_k+d-1}{d-1},$$
and the upper bound is tight.
\end{Theorem}

Taking $e=1$ in Theorem \ref{X-skew-d-partition}, we get Theorem \ref{skew-d-partition}.

\subsection{Bollob\'{a}s systems}

Heged\"{u}s and Frankl \cite{hf} raised the following conjecture on Bollob\'{a}s systems of $d$-partitions of $[n]$.

\begin{Conjecture}{\rm \cite[Conjecture 1]{hf}}\label{conj}
Suppose that $\mathcal{D}=\{(A_i^{(1)},\ldots,A_i^{(d)})\}_{1\leq i\leq m}$ is a Bollob\'{a}s system of $d$-partitions of $[n]$. Then
$$\sum_{i=1}^m\binom{\sum_{r=1}^{d}|A_i^{(r)}|}{|A_i^{(1)}|,\ldots,|A_i^{(d)}|}^{-1}\leq 1.$$
\end{Conjecture}

By Theorem \ref{B-pair}, Conjecture \ref{conj} holds for $d=2$. However, in Section \ref{section-B} we shall show that Conjecture \ref{conj} fails to hold for $d\geq3$ (see Theorem \ref{main-Bollobas}). In particular, we provide a new and tight upper bound for $d=3$ as follows.

\begin{Theorem}\label{main-Bollobas-d=3}
Suppose that $\mathcal{D}=\{(A_i^{(1)},A_i^{(2)},A_i^{(3)})\}_{1\le i\le m}$ is a Bollob\'{a}s system of $3$-partitions of $[n]$. Let $S=\bigcup_{i=1}^m\bigcup_{r=1}^3 A_i^{(r)}$ and $s=|S|$. Then
$$\sum_{i=1}^{m}\binom{|A_i^{(1)}|+|A_i^{(2)}|+|A_i^{(3)}|}{|A_i^{(1)}|,|A_i^{(2)}|,|A_i^{(3)}|}^{-1}\le \left\lfloor\frac{s}{2}\right\rfloor+1,$$
and the upper bound is tight.
\end{Theorem}

\subsection{Strong Bollob\'{a}s systems}

We establish the following theorem, which is a generalization of Lemma 3.1 in \cite{yue}, from $d=2$ to any $d$.

\begin{Theorem}\label{d-for-symmetric}
Suppose that $[n]$ is the disjoint union of some sets $X_1,\ldots,X_e$. Let $\mathcal{D}=\{(A_i^{(1)},\dots$, $A_i^{(d)})\}_{1\le i\le m}$ be a strong Bollob\'{a}s system of $d$-partitions of $[n]$. Let $S=\bigcup_{i=1}^m\bigcup_{r=1}^d A_i^{(r)}$ and $s_k=|S\cap X_k|$ for $k\in[e]$. Then
$$\sum_{i=1}^{m}\left(\prod_{k=1}^{e}\binom{\sum_{r=1}^{d}|A_i^{(r)}\cap X_k| }{|A_i^{(1)}\cap X_k|,\dots,|A_i^{(d)}\cap X_k|}\right)^{-1}\le \min_{l\in [e]}\frac{ \prod_{k=1}^{e}\binom{s_k+d-1}{d-1}}{\binom{s_l+d-1}{d-1}}.$$
\end{Theorem}

The case $e=1$ in Theorem \ref{d-for-symmetric} leads to the following theorem, which shows that Conjecture \ref{conj} holds for strong Bollob\'{a}s systems.

\begin{Theorem}\label{e=1-d-for-symmetric}
Suppose that $\mathcal{D}=\{(A_i^{(1)},\dots,A_i^{(d)})\}_{1\le i\le m}$ is a strong Bollob\'{a}s system of $d$-partitions of $[n]$. Then
$$\sum_{i=1}^{m}\binom{\sum_{r=1}^{d}|A_i^{(r)}|}{|A_i^{(1)}|,\ldots,|A_i^{(d)}|}^{-1}\le1,$$
and the upper bound is tight.
\end{Theorem}

\subsection{Symmetric Bollob\'{a}s systems}

Using the permutation method, Tuza \cite{tuza} proved the following result for symmetric Bollob\'{a}s systems.

\begin{Theorem}{\rm \cite[Proposition 9]{tuza}}\label{symmetric}
Suppose that $\mathcal{D}=\{(A_i^{(1)},\ldots,A_i^{(d)})\}_{1\le i\le m}$ is a symmetric Bollob\'{a}s system of $d$-partitions of $[n]$. Then
$$\sum_{i=1}^{m}\binom{\sum_{r=1}^{d}|A_i^{(r)}|}{|A_i^{(1)}|,\ldots,|A_i^{(d)}|}^{-1}\le1.$$
\end{Theorem}

Since a symmetric Bollob\'{a}s system is a strong Bollob\'{a}s system, Theorem \ref{symmetric} can be seen as a corollary of Theorem \ref{e=1-d-for-symmetric}. The upper bound in Theorem \ref{symmetric} is tight. To show this, we can take the set of all $k$-subsets of $[n]$, $\mathcal{F}=\{F_1,F_2,\ldots,F_m\}$, where $m=\binom{n}{k}$. For $i\in[m]$, let $A_i^{(1)}=F_i$, $A_i^{(2)}=[n]\setminus F_i$, and $A_i^{(r)}=\emptyset$ for $3\leq r\leq d$. Then $\mathcal{D}=\{A_i^{(1)},\ldots,A_i^{(d)}\}_{1\le i\le m}$ is a symmetric Bollob\'{a}s system of $d$-partitions of $[n]$, and
$$\sum_{i=1}^{m}\binom{\sum_{r=1}^{d}|A_i^{(r)}|}{|A_i^{(1)}|,\ldots,|A_i^{(d)}|}^{-1}=\sum_{i=1}^{m}\binom{n}{k}^{-1}=1.$$

The following theorem is a straightforward corollary of Theorem \ref{d-for-symmetric}.

\begin{Theorem}\label{d-symmetric}
Suppose that $[n]$ is the disjoint union of some sets $X_1,\ldots,X_e$. Let $\mathcal{D}= \{(A_i^{(1)},\dots,A_i^{(d)})\}_{1\le i\le m}$ be a symmetric Bollob\'{a}s system of $d$-partitions of $[n]$. Let $s_k=|X_k\cap$ $(\bigcup_{i=1}^m\bigcup_{r=1}^d A_i^{(r)})|$ for $k\in[e]$. Then
$$\sum_{i=1}^{m}\left(\prod_{k=1}^{e}\binom{\sum_{r=1}^{d}|A_i^{(r)}\cap X_k| }{|A_i^{(1)}\cap X_k|,\dots,|A_i^{(d)}\cap X_k|}\right)^{-1}\le \min_{l\in [e]}\frac{ \prod_{k=1}^{e}\binom{s_k+d-1}{d-1}}{\binom{s_l+d-1}{d-1}}.$$
\end{Theorem}

Taking $e=1$ in Theorem \ref{d-symmetric}, we get Theorem \ref{symmetric}.

\subsection{Weak Bollob\'{a}s systems}

A weak Bollob\'{a}s system of $2$-partitions is called \emph{a weak cross-intersecting set-pair system} which was proposed by Tuza \cite{t2} and considered further by Kir\'{a}ly, Nagy, P\'{a}lv\"{o}lgyi and Visontai \cite{knpv}. This paper establishes the following theorem.

\begin{Theorem}\label{improve-weak}
Suppose that $[n]$ is the disjoint union of some sets $X_1,\ldots,X_e$. Let $\mathcal{D}=\{(A_i^{(1)},\ldots,A_i^{(d)})\}_{1\le i\le m}$ be a weak Bollob\'{a}s system of $d$-partitions of $[n]$. Let $S=\bigcup_{i=1}^m\bigcup_{r=1}^d A_i^{(r)}$ and $s_k=|S\cap X_k|$ for $k\in[e]$. Then
$$\sum_{i=1}^{m}\left(\prod_{k=1}^{e}\binom{\sum_{r=1}^{d}|A_i^{(r)}\cap X_k|}{|A_i^{(1)}\cap X_k|,\ldots,|A_{i}^{(d)}\cap X_{k}|}\right)^{-1}\le \prod_{k=1}^{e}\binom{s_k+d-1}{d-1},$$
and the upper bound is tight.
\end{Theorem}

Tuza \cite{tuza} proposed the following inequality to solve various graph-theoretic problems including Ramsey-type problems and coverings problems related to complete graphs.

\begin{Theorem}{\rm \cite[Theorem 1]{tuza}}\label{weak}
Let $p_1,\ldots,p_d$ be arbitrary positive real numbers such that $p_1+\cdots+p_d=1$. Suppose that $\mathcal{D}=\{(A_i^{(1)},\ldots,A_i^{(d)})\}_{1\le i\le m}$ is a weak Bollob\'{a}s system of $d$-partitions of $[n]$. Then
$$\sum_{i=1}^{m}\prod_{r=1}^{d}p_r^{|A_i^{(r)}|}\le1.$$
\end{Theorem}

We will show that the upper bound in Theorem \ref{weak} is tight.
	
\begin{Theorem}\label{weak-sharp}
Let $a_1,\ldots,a_d$ be any given positive integers. Then there exists a weak Bollob\'{a}s system of $d$-partitions of $[n]$, $\mathcal{D}=\{(A_i^{(1)},\ldots,A_i^{(d)})\}_{1\le i\le m}$ with $|A_i^{(r)}|\leq a_r$ for $i\in [m]$ and $r\in [d]$, such that
$$\sum_{i=1}^{m}\prod_{r=1}^{d}p_r^{|A_i^{(r)}|}=1$$
for every $p_1,\ldots,p_d$ satisfying that $0<p_1,\ldots,p_d<1$, and $p_1+\cdots+p_d=1$.
\end{Theorem}

Theorem \ref{weak-sharp} generalizes \cite[Theorem 6]{t2} from $d=2$ to any $d$.

\subsection{Outline of the paper}

In Section \ref{section-pre}, we introduce a fundamental lemma (Lemma \ref{main-lemma}), which plays a crucial role in the proofs of theorems concerning various Bollob\'{a}s systems in the subsequent sections. In Section \ref{section-skew}, we prove Theorem \ref{X-skew-d-partition} on skew Bollob\'{a}s systems. Section \ref{section-B} proves Theorem \ref{main-Bollobas-d=3} on Bollob\'{a}s systems and gives a negative answer to Conjecture \ref{conj}. In Section \ref{section-ss}, we prove Theorems \ref{d-for-symmetric} and \ref{e=1-d-for-symmetric} on strong Bollob\'{a}s systems. In Section \ref{section-weak}, we prove Theorems \ref{improve-weak} and \ref{weak-sharp} on weak Bollob\'{a}s systems. Section \ref{section-concluding} concludes this paper.

\section{Preliminaries}\label{section-pre}

Throughout this paper, for two finite sets of integers $F_1$ and $F_2$, $F_1<F_2$ means $x_1<x_2$ for all $x_1\in F_1$ and $x_2\in F_2$. For any finite set $F$ (including the case where $F$ is the empty set $\emptyset$), we adopt the convention that both $\emptyset<F$ and $F<\emptyset$ hold. For a permutation $\sigma$ of a finite set $X$ and a subset $A\subseteq X$, write $\sigma(A)=\left\{ \sigma(a)\mid a\in A \right\}$.


\begin{Lemma}\label{main-lemma}
Suppose that $[n]$ is the disjoint union of some sets $X_1,\ldots,X_e$. Let $d\geq2$ and $\mathcal{D}=\{(A_i^{(1)},\dots,A_i^{(d)})\}_{1\le i\le m}$ be a family of $d$-partitions of $[n]$. Let $S=\bigcup_{i=1}^m\bigcup_{r=1}^d A_i^{(r)}$. Let $S_k=S\cap X_k$ and $s_k=|S_k|$ for $k\in [e]$. Let $G(S)$ be the set of all permutations of $S$ that keeps $S_k$ invariant for each $k\in [e]$. For $\sigma\in G(S)$, write
$$I_{\sigma}=\{i\in[m]\mid\sigma(A_i^{(1)}\cap X_k)<\sigma(A_i^{(2)}\cap X_k)<\cdots<\sigma(A_i^{(d)}\cap X_k)\ \text{for\ every\ } k\in[e]\}.$$
Then
$$\sum_{i=1}^{m}\left(\prod_{k=1}^{e}s_k!\binom{\sum_{r=1}^{d}|A_i^{(r)}\cap X_k| }{|A_i^{(1)}\cap X_k|,\dots,|A_i^{(d)}\cap X_k|}^{-1}\right)=\sum_{\sigma\in G(S)}|I_{\sigma}|.$$
\end{Lemma}

\begin{proof}
We apply double counting to the set of ordered pairs $(i,\sigma)$ where $i\in [m]$ and $\sigma\in G(S)$ such that $\sigma(A_i^{(1)}\cap X_k)<\sigma(A_i^{(2)}\cap X_k)<\cdots<\sigma(A_i^{(d)}\cap X_k)$ for every $k\in[e]$. Then
\begin{align*}
\sum_{\sigma\in G(S)}|I_{\sigma}|= &\sum_{i=1}^m
\left(\prod_{k=1}^{e} \left(\binom{s_k}{\sum\limits_{r=1}^{d}|A_i^{(r)}\cap X_k|}
\left(\prod_{r=1}^{d} |A_i^{(r)}\cap X_k|!\right)
\left(s_k-\sum_{r=1}^{d}|A_i^{(r)}\cap X_k|\right)!\right)\right)\nonumber\\
= & \sum_{i=1}^{m}\left(\prod_{k=1}^{e}s_k!\binom{\sum_{r=1}^{d}|A_i^{(r)}\cap X_k| }{|A_i^{(1)}\cap X_k|,\dots,|A_i^{(d)}\cap X_k|}^{-1}\right).
\end{align*}
\end{proof}

For fixed integers $d\geq 2$ and $s\in[n]$, denote by $N_{B}(d,s)$ ($N_{skew}(d,s)$, $N_{strong}(d,s)$, $N_{weak}(d,s)$, respectively) the maximum possible size among all Bollob\'{a}s systems (skew Bollob\'{a}s systems, strong Bollob\'{a}s systems, weak Bollob\'{a}s systems, respectively), $\mathcal{D}=\{(A_i^{(1)},\ldots,A_i^{(d)})\}_{1\le i\le m}$, of $d$-partitions of $[n]$ satisfying that $A_i^{(1)}<A_i^{(2)}<\cdots<A_i^{(d)}$ for every $i\in[m]$ and $s=|\bigcup_{i=1}^m\bigcup_{r=1}^d A_i^{(r)}|$.

We introduce the lexicographic ordering for sequences of integers as follows: $(a_1,\ldots,a_d)\le_L(b_1,\ldots,b_d)$ if and only if either the sequences are identical or there exists an index $l\in[d]$ such that $a_i=b_i$ for all $1\le i<l$, and $a_l<b_l$. This ordering compares sequences element-wise from left to right, with the first differing element determining the order.

\begin{Proposition}\label{skew-n}
$$N_{skew}(d,s)=\binom{s+d-1}{d-1}.$$
\end{Proposition}

\begin{proof}
Let $\mathcal{D}=\{(A_i^{(1)},\ldots,A_i^{(d)})\}_{1\le i\le m}$ be a family of $d$-partitions of $[n]$ (not necessarily a skew Bollob\'{a}s system) such that $A_i^{(1)}<A_i^{(2)}<\cdots<A_i^{(d)}$ for $i\in[m]$, and $|\bigcup_{i=1}^m\bigcup_{r=1}^d A_i^{(r)}|=s$. Let $S=\bigcup_{i=1}^m\bigcup_{r=1}^d A_i^{(r)}$. For each $j\in[m],$ we can naturally add the elements in $S\setminus \bigcup_{r=1}^d A_j^{(r)}$ to $A_j^{(1)},\ldots,A_j^{(d)}$ to produce a family of full $d$-partitions of $S$, $\mathcal{D'}=\{(B_i^{(1)},\ldots,B_i^{(d)})\}_{1\le i\le m}$, such that $B_i^{(1)}<B_i^{(2)}<\cdots<B_i^{(d)}$ for $i\in[m]$. There are $\binom{s+d-1}{d-1}$ ordered partitions $(a_1,\ldots,a_d)$ of the integer $s$ where $a_1,\ldots,a_d$ are non-negative integers satisfying $a_1+\cdots+a_d=s$. So $N_{skew}(d,s)\leq\binom{s+d-1}{d-1}.$

Let $\mathcal{A}=\{(A_i^{(1)},\ldots,A_i^{(d)})\}_{1\le i\le m}$ consist of all full $d$-partitions of $[s]$ such that $A_i^{(1)}<A_i^{(2)}<\cdots<A_i^{(d)}$ for $i\in[m]$, and $(|A_j^{(1)}|,\ldots,|A_j^{(d)}|)\le_L(|A_i^{(1)}|,\ldots,|A_i^{(d)}|)$ for all $1\le i<j\le m$. Then $\mathcal{A}$ is a skew Bollob\'{a}s system and $|\mathcal{A}|=m=\binom{s+d-1}{d-1}.$ Therefore, $N_{skew}(d,s)=\binom{s+d-1}{d-1}.$
\end{proof}

\begin{Remark}
Let $\mathcal{D}=\{(A_i^{(1)},\ldots,A_i^{(d)})\}_{1\leq i\leq m}$ be a skew Bollob\'{a}s system of $d$-partitions of $[n]$. Apply Lemma $\ref{main-lemma}$ with $e=1$ and write $s=|\bigcup_{i=1}^m\bigcup_{r=1}^d A_i^{(r)}|$. Since $|I_{\sigma}|\leq N_{skew}(d,s)$, it follows from Lemma $\ref{main-lemma}$ and Proposition $\ref{skew-n}$ that
$$s!\sum_{i=1}^{m} \binom{\sum_{r=1}^{d}|A_i^{(r)}|}{|A_{i}^{(1)}|,\ldots,|A_{i}^{(d)}|}^{-1}
=\sum_{\sigma\in G(S)}|I_{\sigma}|
\le s!\binom{s+d-1}{d-1}\le s!\binom{n+d-1}{d-1}.$$
Dividing both sides by $s!$ yields Theorem $\ref{skew-d-partition}$.
\end{Remark}

\begin{Proposition}\label{strong-n}
$$N_{strong}(d,s)=1.$$
\end{Proposition}

\begin{proof}
Let $\mathcal{D}=\{(A_i^{(1)},\ldots,A_i^{(d)})\}_{1\le i\le m}$ be a strong Bollob\'{a}s system of $d$-partitions of $[n]$ such that $A_i^{(1)}<A_i^{(2)}<\cdots<A_i^{(d)}$ for $i\in[m]$.
Suppose on the contradiction that there exists $i<j$ such that $A_i^{(u_1)}\cap A_j^{(v_2)}\neq\emptyset$ and $A_i^{(u_2)}\cap A_j^{(v_1)}\neq\emptyset$ for some $u_1<u_2$ and $v_1<v_2$. Then there exist $x\in A_i^{(u_1)}\cap A_j^{(v_2)}$ and $y\in A_i^{(u_2)}\cap A_j^{(v_1)}$. Since $A_i^{(u_1)}<A_i^{(u_2)}$ we have $x<y$. Since $A_j^{(v_1)}<A_j^{(v_2)}$, we have $y<x$, a contradiction.
\end{proof}

\begin{Proposition}\label{lem:weak-n}
$$N_{weak}(d,s)=N_{skew}(d,s).$$
\end{Proposition}

\begin{proof}
Since any skew Bollob\'{a}s system is a weak Bollob\'{a}s system, we have $N_{skew}(d,s)\leq N_{weak}(d,s)$. On the other hand, by the first paragraph of the proof of Proposition \ref{skew-n}, we have $N_{weak}(d,s)\leq \binom{s+d-1}{d-1}=N_{skew}(d,s)$. Therefore, $N_{weak}(d,s)=N_{skew}(d,s)$.
\end{proof}


\section{Skew Bollob\'{a}s systems}\label{section-skew}

This section is devoted to proving Theorem \ref{X-skew-d-partition}.

\begin{proof}[Proof of Theorem $\ref{X-skew-d-partition}$.]
Apply Lemma $\ref{main-lemma}$. We have
\begin{align*}
\sum_{i=1}^{m}\left(\prod_{k=1}^{e}s_k!\binom{\sum_{r=1}^{d}|A_i^{(r)}\cap X_k| }{|A_i^{(1)}\cap X_k|,\dots,|A_i^{(d)}\cap X_k|}^{-1}\right)
=
\sum_{\sigma\in G(S)}|I_{\sigma}|\leq \left(\prod_{k=1}^{e}s_k!\right) \left(\prod_{k=1}^{e}\binom{s_k+d-1}{d-1}\right).
\end{align*}
Dividing both sides by $\prod_{k=1}^{e}s_k!$ yields the desired inequality. To prove that the upper bound is tight, we need to construct a skew Bollob\'{a}s system of $d$-partitions of $[n]$, $\mathcal{D}=\{(A_i^{(1)},\ldots,$ $A_i^{(d)})\}_{1\le i\le m}$, such that
\begin{align}\label{eq:2}
\sum_{i=1}^{m}\left(\prod_{k=1}^{e}\binom{\sum_{r=1}^{d}|A_i^{(r)}\cap X_k|}{|A_i^{(1)}\cap X_k|,\ldots,|A_{i}^{(d)}\cap X_{k}|}\right)^{-1}
= \prod_{k=1}^{e}\binom{s_k+d-1}{d-1}.
\end{align}

Let $\mathcal{D}=\{(A_i^{(1)},\ldots,A_i^{(d)})\}_{1\le i\le d^n}$ be the set of all full $d$-partitions of $[n]$ such that $(|A_j^{(1)}|,\ldots$, $|A_j^{(d)}|)\le_L(|A_l^{(1)}|,\ldots,|A_l^{(d)}|)$ for all $1\le l<j\le d^n$. It was shown by Heged\"{u}s and Frankl in \cite[Example 2]{hf} that $\mathcal{D}$ is a skew Bollob\'{a}s system. Suppose that $[n]$ is partitioned into disjoint sets $X_1,\dots,X_e$. Write $s_k=|X_k|$ for $k\in[e]$. We will show that $\mathcal D$ satisfies \eqref{eq:2}.

For $i\in [d^n]$, $r\in[d]$ and $k\in[e]$, write $\delta_k^{(r)}=|A_i^{(r)}\cap X_k|$. We call $((\delta_1^{(1)},\ldots,\delta_1^{(d)}),(\delta_2^{(1)},\ldots,\delta_2^{(d)}),$ $\ldots,(\delta_e^{(1)},\ldots,\delta_e^{(d)}))$ the {\em type} of $(A_i^{(1)},\ldots,A_i^{(d)})$. There are ${\textstyle\prod_{k=1}^{e}}\binom{s_k+d-1}{d-1}$ different types contributed by $\mathcal D$ (note that $\mathcal D$ is full), because there are $\binom{s_k+d-1}{d-1}$ ordered partitions $(\delta_k^{(1)},\ldots,\delta_k^{(d)})$ of $s_k$ such that $\delta_k^{(1)},\ldots,\delta_k^{(d)}$ are non-negative integers satisfying $\delta_k^{(1)}+\cdots+\delta_k^{(d)}=s_k$.

For each $(A_i^{(1)},\ldots, A_i^{(d)})\in \mathcal{D}$ with $i\in[d^n]$, its contribution to the sum on the left-hand side of \eqref{eq:2} is ${\textstyle\prod_{k=1}^{e}}\binom{s_k}{\delta_k^{(1)},\ldots,\delta_k^{(d)}}^{-1}$. On the other hand, there are altogether ${\textstyle\prod_{k=1}^{e}}\binom{s_k}{\delta_k^{(1)},\ldots,\delta_k^{(d)}}$ tuples $(A_j^{(1)},\ldots,A_j^{(d)})$'s in $\mathcal D$ that have the same type as $(A_i^{(1)},\ldots,A_i^{(d)})$. So the total contribution of a type of $(A_i^{(1)},\ldots,A_i^{(d)})$'s to the sum is exactly $1$. Since there are ${\textstyle\prod_{k=1}^{e}}\binom{s_k+d-1}{d-1}$ different types contributed by $\mathcal D$, after summing up the contributions, the left-hand side of \eqref{eq:2} equals ${\textstyle\prod_{k=1}^{e}}\binom{s_k+d-1}{d-1}$.
\end{proof}

\section{Disproof of Conjecture \ref{conj} on Bollob\'{a}s systems}\label{section-B}

In this section we prove Theorem \ref{main-Bollobas-d=3} and give a negative answer to Conjecture \ref{conj}.

\begin{Theorem}\label{main-Bollobas}
Suppose that $\mathcal{D}=\{(A_i^{(1)},\ldots, A_i^{(d)})\}_{1\le i\le m}$ is a Bollob\'{a}s system of $d$-partitions of $[n]$. Let $S=\bigcup_{i=1}^m\bigcup_{r=1}^d A_i^{(r)}$ and $s=|S|$. Then
$$\sum_{i=1}^{m}\binom{\sum_{r=1}^{d}|A_i^{(r)}|}{|A_i^{(1)}|,\dots,|A_i^{(d)}|}^{-1}\le N_B(d,s),$$
and the upper bound is tight.
\end{Theorem}

\begin{proof}
Applying Lemma \ref{main-lemma} with $e=1$, we have
$$s!\sum_{i=1}^{m}\binom{\sum_{r=1}^{d}|A_i^{(r)}|}{|A_i^{(1)}|,\dots,|A_i^{(d)}|}^{-1}=\sum_{\sigma\in G(S)}|I_{\sigma}|\le s!N_B(d,s).$$
Dividing  both sides by $s!$ yields the desired inequality. To prove that the upper bound is tight, we shall construct a full Bollob\'{a}s system of $d$-partitions of $[s]$, $\mathcal{D}=\{(A_i^{(1)},\ldots,A_i^{(d)})\}_{1\le i\le m}$, such that
\begin{align}\label{eq:1}
\sum_{i=1}^{m}\binom{\sum_{r=1}^{d}|A_i^{(r)}|}{|A_i^{(1)}|,\dots,|A_i^{(d)}|}^{-1}= N_B(d,s).
\end{align}

According to the definition of $N_B(d,s)$, there exists a Bollob\'{a}s system of $d$-partitions of $[s]$, $\mathcal{B}=\{(B_l^{(1)},\ldots, B_l^{(d)})\}_{1\le l\le N_B(d,s)}$, such that $B_l^{(1)}<B_l^{(2)}<\cdots<B_l^{(d)}$ for every $l\in[N_B(d,s)]$. Without loss of generality, we assume that $\mathcal B$ is full. We call $(|B_l^{(1)}|,\ldots, |B_l^{(d)}|)$ the type of $(B_l^{(1)},\ldots, B_l^{(d)})$. Since $\mathcal B$ is full, all the members in $\mathcal{B}$ have different types. For each $(B_l^{(1)},\ldots, B_l^{(d)})\in \mathcal{B}$, $l\in[N_B(d,s)]$, let $\mathcal{B}_l$ be the set of all full $d$-partitions of $[s]$ which have the same type as $(B_l^{(1)},\ldots, B_l^{(d)})$. Let $\mathcal{D}=\bigcup_{l=1}^{N_B(d,s)}\mathcal{B}_l$. We will demonstrate that $\mathcal D$ is a full Bollob\'{a}s system satisfying \eqref{eq:1} (here every element in $\mathcal D$ can be renamed by some $(A_i^{(1)},\ldots,A_i^{(d)})$).

Since for each $1\leq l\leq N_B(d,s)$,
$$|\mathcal{B}_l|=\binom{s}{|B_l^{(1)}|,\ldots,|B_l^{(d)}|}=\binom{\sum_{r=1}^{d}|B_l^{(r)}|}{|B_l^{(1)}|,\ldots,|B_l^{(d)}|},$$
we have
$$\sum_{l=1}^{N_B(d,s)}|\mathcal{B}_l|
\binom{\sum_{r=1}^{d}|B_l^{(r)}|}{|B_l^{(1)}|,\dots,|B_l^{(d)}|}^{-1}=
\sum_{l=1}^{N_B(d,s)}1=N_B(d,s).$$
Therefore, \eqref{eq:1} holds. It remains to show that $\mathcal{D}$ is a Bollob\'{a}s system of full $d$-partitions of $[s]$. In the following we assume that $(A_j^{(1)},\dots,A_j^{(d)})$ and $(A_k^{(1)},\dots,A_k^{(d)})$ with $j\ne k$ are both from $\mathcal D$ (note that every element in $\mathcal D$ has be renamed by some $(A_i^{(1)},\ldots,A_i^{(d)})$).

Case $1$: Suppose that $(A_j^{(1)},\ldots,A_j^{(d)})$ and $(A_k^{(1)},\ldots,A_k^{(d)})$ are with the same type. That is, $|A_j^{(r)}|=|A_k^{(r)}|$ for every $r\in[d]$. Then there exists the smallest integer $u$ in $[d]$ such that $A_j^{(u)}\ne A_k^{(u)}$. So $A_j^{(u)}\cap A_k^{(q_1)}\ne \emptyset$ and $A_k^{(u)}\cap A_j^{(q_2)}\ne \emptyset$ for some $q_1$ and $q_2$ such that $1\leq u<q_1,q_2\leq d$.
	
Case $2$: Suppose that $(A_j^{(1)},\ldots,A_j^{(d)})$ and $(A_k^{(1)},\ldots,A_k^{(d)})$ are with different types. Let $v\in[d]$ be the smallest integer such that $|A_j^{(v)}|\ne|A_k^{(v)}|$. Without loss of generality, we assume that $|A_j^{(v)}|>|A_k^{(v)}|$. If there exists some $u<v$ such that $A_j^{(u)}\ne A_k^{(u)}$ and $u$ is the smallest integer satisfying this condition, we get the desired conclusion with the same proof as that for Case $1$. Otherwise, $A_j^{(u)}=A_k^{(u)}$ holds for every $u<v$. We deal with this case below.

Since $|A_j^{(v)}|>|A_k^{(v)}|$, there exists a $q_1>v$ such that $A_j^{(v)}\cap A_k^{(q_1)}\ne\emptyset$. If there exists some $w>v$ such that $|A_j^{(v)}|+|A_j^{(v+1)}|+\cdots+|A_j^{(w)}|<|A_k^{(v)}|+|A_k^{(v+1)}|+\cdots+|A_k^{(w)}|$, then there exist $v\leq p\le w$ and $w<q_2\le d$ such that $A_j^{(q_2)}\cap A_k^{(p)}\ne\emptyset$. Otherwise, $|A_j^{(v)}|+|A_j^{(v+1)}|+\cdots+|A_j^{(w)}|\geq|A_k^{(v)}|+|A_k^{(v+1)}|+\cdots+|A_k^{(w)}|$ holds for all $v\leq w\leq d$. It leads to that there are $(B_j^{(1)},\dots,B_j^{(d)}), (B_k^{(1)},\dots,B_k^{(d)})\in \mathcal{B}$, $j\ne k$, satisfying that $\sum_{t=1}^{r}|B_j^{(t)}|\geq\sum_{t=1}^{r}|B_k^{(t)}|$ for all $r\in[d]$. A contradiction occurs according to the definition of $\mathcal{B}$.
\end{proof}

It is easy to see that $N_B(2,s)=1$. Thus applying Theorem \ref{main-Bollobas} with $d=2$, we obtain Theorem \ref{B-pair}. Next we determine the value of $N_B(3,s)$.

\begin{Proposition}\label{Bollobas-d=3}
$$N_B(3,s)=\left\lfloor\frac{s}{2}\right\rfloor+1.$$
\end{Proposition}

\begin{proof}
For $l\in[\lfloor\frac{s}{2}\rfloor+1]$, let $B_l^{(1)}=[1,l-1]$, $B_l^{(2)}=[l,s-l+1]$, and $B_l^{(3)}=[s-l+2,s]$. It is readily checked that $\mathcal{B}=\{(B_l^{(1)},B_l^{(2)},B_l^{(3)})\}_{1\le l\le \lfloor\frac{s}{2}\rfloor+1}$ forms a Bollob\'{a}s system of full $3$-partitions of $[s]$ such that $B_l^{(1)}<B_l^{(2)}<B_l^{(3)}$ for every $l\in[\lfloor\frac{s}{2}\rfloor+1]$. Hence, $N_B(3,s)\geq \lfloor\frac{s}{2}\rfloor+1$.

Let $\mathcal{D}=\{(A_i^{(1)},A_i^{(2)},A_i^{(3)})\}_{1\le i\le m}$ be a Bollob\'{a}s system of $3$-partitions of $[s]$ such that $A_i^{(1)}<A_i^{(2)}<A_i^{(3)}$ for every $i\in[m]$. Without loss of generality, assume that $\mathcal D$ is full. Suppose that $(A_j^{(1)},A_j^{(2)},A_j^{(3)})$ and $(A_k^{(1)},A_k^{(2)},A_k^{(3)})$ with $j\ne k$ are both from $\mathcal{D}$. Then all the possibilities are either (I) $A_j^{(1)}\cap A_k^{(2)}\ne \emptyset$ and $A_j^{(3)}\cap A_k^{(2)}\ne \emptyset$, or (II) $A_k^{(1)}\cap A_j^{(2)}\ne \emptyset$ and $A_k^{(3)}\cap A_j^{(2)}\ne \emptyset$.

Case $1$: $|A_j^{(2)}|=|A_k^{(2)}|$. If $A_j^{(1)}\cap A_k^{(2)}\ne \emptyset$, then $|A_j^{(1)}|\geq |A_k^{(1)}|+1$ and thus $|A_j^{(3)}|\leq |A_k^{(3)}|-1$, which implies that $A_j^{(3)}\cap A_k^{(2)}=\emptyset$. Hence (I) is impossible. If $A_k^{(1)}\cap A_j^{(2)}\ne \emptyset$, then $|A_k^{(1)}|\geq |A_j^{(1)}|+1$ and thus $|A_k^{(3)}|\leq |A_j^{(3)}|-1$, which implies that $A_k^{(3)}\cap A_j^{(2)}=\emptyset$. Hence (II) is also impossible, a contradiction. So for each $0\leq x\leq s$, there is at most one $(A_i^{(1)},A_i^{(2)},A_i^{(3)})\in \mathcal{D}$ with $|A_i^{(2)}|=x$.

Case $2$: $|A_j^{(2)}|+1=|A_k^{(2)}|$. If $A_j^{(1)}\cap A_k^{(2)}\ne \emptyset$, then $|A_j^{(1)}|\geq |A_k^{(1)}|+1$ and thus $|A_j^{(3)}|\leq |A_k^{(3)}|$, which implies that $A_j^{(3)}\cap A_k^{(2)}=\emptyset$. Hence (I) is impossible. If $A_k^{(1)}\cap A_j^{(2)}\ne \emptyset$, then $|A_k^{(1)}|\geq |A_j^{(1)}|+1$ and thus $|A_k^{(3)}|\leq |A_j^{(3)}|-2$, which implies that $A_k^{(3)}\cap A_j^{(2)}=\emptyset$. Hence (II) is impossible, a contradiction. So for each
$0\leq x\leq s-1$, at most one $3$-partition between $(A_j^{(1)},A_j^{(2)},A_j^{(3)})$ and $(A_k^{(1)},A_k^{(2)},A_k^{(3)})$ with $|A_j^{(2)}|=x$ and $|A_k^{(2)}|=x+1$ belongs to $\mathcal{D}$.

Combining Cases 1 and 2, we have $m\leq \lfloor\frac{s}{2}\rfloor+1$. Therefore, $N_B(3,s)=\lfloor\frac{s}{2}\rfloor+1$.
\end{proof}

\begin{proof}[Proof of Theorem $\ref{main-Bollobas-d=3}$.]
Apply Theorem \ref{main-Bollobas} and Proposition \ref{Bollobas-d=3} to complete the proof.
\end{proof}

Theorem \ref{main-Bollobas-d=3} implies that Conjecture \ref{conj} does not hold when $d=3$. Let $d\geq 4$ and $\mathcal{D}=\{(A_i^{(1)},\ldots, A_i^{(d)})\}_{1\le i\le m}$ be a Bollob\'{a}s system of $d$-partitions of $[n]$. Let $S=\bigcup_{i=1}^m\bigcup_{r=1}^d A_i^{(r)}$ and $s=|S|$. By Theorem \ref{main-Bollobas},
$$\sum_{i=1}^{m}\binom{\sum_{r=1}^{d}|A_i^{(r)}|}{|A_i^{(1)}|,\dots,|A_i^{(d)}|}^{-1}\le N_B(d,s).$$
Since $N_B(d,s)\geq N_B(3,s)=\lfloor\frac{s}{2}\rfloor+1>1$ for any $d\geq 4$, Conjecture \ref{conj}  would not be true for any $d\geq 4$.

\section{Strong Bollob\'{a}s systems}\label{section-ss}

This section gives proofs for Theorems \ref{d-for-symmetric} and \ref{e=1-d-for-symmetric}.

\begin{proof}[Proof of Theorem $\ref{d-for-symmetric}$.]
Applying Lemma $\ref{main-lemma}$, we have
\begin{align}\label{eq:4}
\sum_{i=1}^{m}\left(\prod_{k=1}^{e}s_k!\binom{\sum_{r=1}^{d}|A_i^{(r)}\cap X_k| }{|A_i^{(1)}\cap X_k|,\dots,|A_i^{(d)}\cap X_k|}^{-1}\right)=
\sum_{\sigma\in G(S)}|I_{\sigma}|.
\end{align}
We will show that
\begin{align}\label{eq:3}
\sum_{\sigma\in G(S)}|I_{\sigma}|\leq \left(\prod_{k=1}^{e}s_k!\right) \min_{l\in [e]}\frac{ \prod_{k=1}^{e}\binom{s_k+d-1}{d-1}}{\binom{s_l+d-1}{d-1}}.
\end{align}
Combining \eqref{eq:4} and \eqref{eq:3}, we can prove Theorem \ref{d-for-symmetric}.

To prove \eqref{eq:3}, we use induction on $e$. For $e=1$, since $N_{strong}(d,s_1)=1$ by Proposition \ref{strong-n}, we have
$$\sum_{\sigma\in G(S)}|I_{\sigma}|\leq s_1!N_{strong}(d,s_1)=s_1!.$$
Assume that \eqref{eq:3} holds for the case $e-1$ with $e\geq2$, and let us prove it for $e$. Suppose that $[n]$ is the disjoint union of $X_1,\ldots,X_e$. Let $\mathcal{D}=\{(A_i^{(1)},\dots$, $A_i^{(d)})\}_{1\le i\le m}$ be a strong Bollob\'{a}s system of $d$-partitions of $[n]$. Given any $i\in [m]$ and $l^*\in[e]$, let
\begin{align}
\mathcal{D}_{i,l^*}=&\{(A_j^{(1)}\cap ([n]\setminus X_{l^*}),\dots,A_j^{(d)}\cap ([n]\setminus X_{l^*}))\mid \nonumber\\
&~~~~~~(A_j^{(1)}\cap X_{l^*},\dots,A_j^{(d)}\cap X_{l^*})=(A_i^{(1)}\cap X_{l^*},\dots,A_i^{(d)}\cap X_{l^*}), j\in[m]\}.\nonumber
\end{align}
Then $\mathcal{D}_{i,l^*}$ is also a strong Bollob\'{a}s system of $d$-partitions of $[n]$. Let $R_{l^*}=(\bigcup_{i=1}^m\bigcup_{r=1}^d A_i^{(r)})\setminus X_{l^*}$. Applying induction hypothesis to $\mathcal{D}_{i,l^*}$ gives us
\begin{align*}
\sum_{\sigma\in G(R_{l^*})}|I_{\sigma}|&\leq \left(\prod_{k\in[e]\setminus\{l^*\}}s_k!\right)\min_{l\in [e]\setminus\{l^*\}}\frac{ \prod_{k\in[e]\setminus\{l^*\}}\binom{s_k+d-1}{d-1}}{\binom{s_l+d-1}{d-1}}.
\end{align*}
Therefore,
\begin{align}\label{eq:5}
\sum_{\sigma\in G(S)}|I_{\sigma}|&\leq \left(s_{l^*}!\binom{s_{l^*}+d-1}{d-1}\right)\cdot
\left(\prod_{k\in[e]\setminus\{l^*\}}s_k!\right)\min_{l\in [e]\setminus\{l^*\}}\frac{ \prod_{k\in[e]\setminus\{l^*\}}\binom{s_k+d-1}{d-1}}{\binom{s_l+d-1}{d-1}}\nonumber
\\
&=\left(\prod_{k\in[e]}s_k!\right)\min_{l\in [e]\setminus\{l^*\}}\frac{ \prod_{k\in[e]}\binom{s_k+d-1}{d-1}}{\binom{s_l+d-1}{d-1}}.
\end{align}
Since \eqref{eq:5} holds for any given $l^*\in [e]$, we obtain \eqref{eq:3}.
\end{proof}

\begin{proof}[Proof of Theorem $\ref{e=1-d-for-symmetric}$.]
Apply Theorem \ref{d-for-symmetric} with $e=1$. It remains to show that the upper bound is tight.

Let $\mathcal{A}=\{(A_i^{(1)},\dots,A_i^{(n)})\}_{1\le i\le m}$ be the set of all full $n$-partitions of $[n]$ such that $|A_i^{(r)}|=1$ for $i\in[m]$ and $r\in[n]$, and so $m=n!$. For each $1\leq i<j\leq m$, suppose $l$ is the smallest integer in $[n]$ such that $A_i^{(l)}\neq A_j^{(l)}$. Let $u_1=v_1=l$. Then there exist $u_2>u_1$ and $v_2>v_1$ such that $A_i^{(u_1)}\cap A_j^{(v_2)}\neq \emptyset$ and $A_i^{(u_2)}\cap A_j^{(v_1)}\neq \emptyset$. Thus $\mathcal{A}$ is a strong Bollob\'{a}s system of $n$-partitions of $[n]$ and we have
$$\sum_{i=1}^{m}\binom{\sum_{r=1}^{n}|A_i^{(r)}|}{|A_i^{(1)}|,\ldots,|A_i^{(d)}|}^{-1}=\sum_{i=1}^{n!}\binom{n}{1,\ldots,1}^{-1}=1.$$
This provides an example of a strong Bollob\'{a}s system attaining the upper bound.
\end{proof}

The case $e=2$ in Theorem \ref{d-for-symmetric} implies
	
\begin{equation}
\begin{aligned}
&\sum_{i=1}^{m}\left ( \binom{\sum_{r=1}^{d}|A_i^{{(r)} }\cap X_1| }{|A_i^{(1)}\cap X_1|,\ldots,|A_i^{(d)}\cap X_1|}\binom{\sum_{r=1}^{d}|A_i^{(r)}\cap X_2| }{|A_i^{(1)}\cap X_2|,\ldots,|A_i^{(d)}\cap X_2|}\right)^{-1}\\
\le &\min  \left\{\binom{s_1+d-1}{d-1},\binom{s_2+d-1}{d-1}\right\}\le \binom{\left\lfloor\frac{n}{2}\right\rfloor +d-1}{d-1},\nonumber
\end{aligned}
\end{equation}
and so we get the following result.
	
\begin{Theorem}\label{e=2-symmetric}
Suppose that $[n]$ is the union of two disjoint sets $X_1$ and $X_2$. Let $\mathcal{D}=\{(A_i^{(1)},\ldots,A_i^{(d)})\}_{1\le i\le m}$ be a strong Bollob\'{a}s system of $d$-partitions of $[n]$. Then
$$\sum_{i=1}^{m}\left ( \binom{\sum_{r=1}^{d}|A_i^{{(r)} }\cap X_1|}{|A_i^{(1)}\cap X_1|,\dots,|A_i^{(d)}\cap X_1|}\binom{\sum_{r=1}^{d}|A_i^{(r)}\cap X_2| }{|A_i^{(1)}\cap X_2|,\dots,|A_i^{(d)}\cap X_2|}\right)^{-1}\le \binom{\left \lfloor \frac{n}{2}  \right\rfloor +d-1}{d-1}.$$
\end{Theorem}


\section{Weak Bollob\'{a}s systems}\label{section-weak}

In this section we prove Theorems \ref{improve-weak} and \ref{weak-sharp}.

\begin{proof}[Proof of Theorem $\ref{improve-weak}$.]
Applying Lemma $\ref{main-lemma}$, we have
\begin{align*}
\sum_{i=1}^{m}\left(\prod_{k=1}^{e}s_k!\binom{\sum_{r=1}^{d}|A_i^{(r)}\cap X_k|}{|A_i^{(1)}\cap X_k|,\ldots,|A_{i}^{(d)}\cap X_{k}|}\right)^{-1}=
\sum_{\sigma\in G(S)}|I_{\sigma}| \le \left(\prod_{k=1}^{e}s_k!\right)\left(\prod_{k=1}^{e}\binom{s_k+d-1}{d-1}\right).
\end{align*}
Dividing both sides by $\prod_{k=1}^{e}s_k!$ yields the desired upper bound. Furthermore, since a skew Bollob\'{a}s system is a weak Bollob\'{a}s system, the upper bound is tight by Theorem \ref{X-skew-d-partition}.
\end{proof}

\begin{proof}[Proof of Theorem $\ref{weak-sharp}$.]
For any given positive integers $a_1,\ldots,a_d$, let $n=a_1+a_2+\cdots+a_d-1$. We define a family $\mathcal{D}$ of $d$-partitions of $[n]$ as follows:
\begin{align*}
\mathcal{D}=&(\mathcal{D}_{a_1,a_2,\ldots,a_{d-1},0}\cup \mathcal{D}_{a_1,a_2,\ldots,a_{d-1},1}\cup\cdots\cup  \mathcal{D}_{a_1,a_2,\ldots,a_{d-1},a_d-1})\cup\nonumber\\
&~~~~~~(\mathcal{D}_{a_1,\ldots,a_{d-2},0,a_d}\cup\mathcal{D}_{a_1,\ldots,a_{d-2},1,a_d}\cup\cdots\cup \mathcal{D}_{a_1,\ldots,a_{d-2},a_{d-1}-1,a_d}) \cup\nonumber\\
&~~~~~~\cdots\cup(\mathcal{D}_{0,a_2,\ldots,a_d}\cup
\mathcal{D}_{1,a_2,\ldots,a_d}\cup\cdots\cup\mathcal{D}_{a_1-1,a_2,\ldots,a_d}),
\end{align*}
where for every $0\le k_d\le a_d-1$,
\begin{align*}
\mathcal{D}_{a_1,a_2,\ldots,a_{d-1},k_d}=&\{(A^{(1)},\ldots, A^{(d)}) \mid A^{(1)},\ldots,A^{(d)} \text{ form a disjoint union that is } [n+1-a_d+k_d],\\
& ~~~~~~ |A^{(r)}|=a_r \text{ for } r\in[d]\setminus\{d\}, |A^{(d)}|=k_d, \\
& ~~~~~~ \text{ and } n+1-a_d+k_d \in \cup_{r\in[d]\setminus\{d\}} A^{(r)}\},
\end{align*}
for every $0\le k_{d-1}\le a_{d-1}-1$,
\begin{align*}
\mathcal{D}_{a_1,\ldots,a_{d-2},k_{d-1},a_d}=&\{(A^{(1)},\ldots, A^{(d)}) \mid A^{(1)},\ldots,A^{(d)} \text{ form a disjoint union that is }  \\
& ~~~~[n+1-a_{d-1}+k_{d-1}], |A^{(r)}|=a_r \text{ for } r\in[d]\setminus\{d-1\}, |A^{(d-1)}|=k_{d-1}, \\
& ~~~~ \text{ and } n+1-a_{d-1}+k_{d-1}\in \cup_{r\in[d]\setminus\{d-1\}}A^{(r)} \},
\end{align*}
$\cdots$, and for every $0\le k_1\le a_1-1$,
\begin{align*}
\mathcal{D}_{k_1,a_2,\ldots,a_d}=&\{(A^{(1)},\ldots, A^{(d)}) \mid A^{(1)},\ldots,A^{(d)} \text{ form a disjoint union that is } [n+1-a_1+k_1],\\
& ~~~~~~ |A^{(r)}|=a_r \text{ for } r\in[d]\setminus\{1\}, |A^{(1)}|=k_1, \\
& ~~~~~~  \text{ and } n+1-a_1+k_1\in \cup_{r\in[d]\setminus\{1\}}A^{(r)}\}.
\end{align*}

\begin{Claim}
$\mathcal{D}$ is a weak Bollob\'{a}s system of $d$-partitions of $[n]$.
\end{Claim}
\begin{proof}
Let $(A_i^{(1)},\ldots,A_i^{(d)})$ and $(A_j^{(1)},\ldots,A_j^{(d)})$ be two members of $\mathcal{D}$ with $i\ne j$. We shall show that there exist $1\le p < q\le d$ such that $A_i^{(p)}\cap A_j^{(q)}\ne\emptyset$ or $A_j^{(p)}\cap A_i^{(q)}\ne\emptyset$.

Case 1: Suppose that $(A_i^{(1)},\ldots,A_i^{(d)})$ and $(A_j^{(1)},\ldots,A_j^{(d)})\in \mathcal{D}_{a_1,\ldots,k_u,\ldots,a_d}$ for a fixed $k_u$ with $u\in[d]$. Since $\bigcup_{r=1}^{d}A_i^{(r)}=\bigcup_{r=1}^{d}A_j^{(r)}$ and $|A_i^{(r)}|=|A_j^{(r)}|$ for every $r\in[d]$, it is easy to see that $\mathcal{D}_{a_1,\ldots,k_u,\ldots,a_d}$ is a strong Bollob\'{a}s system of $d$-partitions of $[n]$. So there exist $1\le p < q\le d$ such that $A_i^{(p)}\cap A_j^{(q)}\ne\emptyset$ or $A_j^{(p)}\cap A_i^{(q)}\ne\emptyset$.

Case 2: Suppose that $(A_i^{(1)},\ldots,A_i^{(d)})\in \mathcal{D}_{a_1,\ldots,k_u=u_1,\ldots,a_d}$, $(A_j^{(1)},\ldots,A_j^{(d)})\in\mathcal{D}_{a_1,\ldots,k_u=u_2,\ldots,a_d}$, where $0\leq u_1<u_2\leq a_u-1$ and $u\in[d]$. By the definition of $\mathcal{D}_{a_1,\ldots,k_u=u_2,\ldots,a_d}$, there exists some $v\ne u$ such that $n+1-a_u+u_2 \in A_j^{(v)}$. By the definition of $\mathcal{D}_{a_1,\ldots,k_u=u_1,\ldots,a_d}$, we have $A_{i}^{(v)}\subseteq [n+1-a_u+u_1]$, which concludes to $n+1-a_u+u_2 \notin A_{i}^{(v)}$ due to $u_1<u_2$. So $A_i^{(v)}\neq A_j^{(v)}$. Note that $|A_i^{(v)}|=|A_j^{(v)}|$ and $A_i^{(v)}\subset [n+1-a_u+u_2]=\bigcup_{r=1}^{d}A_j^{(r)}$. Then we get $A_i^{(v)}\cap A_j^{(q)}\ne\emptyset$ for some $q\in[d]\setminus\{v\}$.

Case 3: Suppose that $(A_i^{(1)},\ldots,A_i^{(d)})\in\mathcal{D}_{a_1,\ldots,k_u,\ldots,a_d}$ and $(A_j^{(1)},\ldots,A_j^{(d)})\in\mathcal{D}_{a_1,\ldots,k_v,\ldots,a_d}$, where $1\leq v<u\leq d$. Since $|A_i^{(v)}|+(|A_j^{(1)}|+\cdots+|A_j^{(v-1)}|+|A_j^{(v+1)}|+\cdots+|A_j^{(d)}|)=a_1+\cdots+a_d=n+1>n$, we have $A_i^{(v)}\cap(A_j^{(1)}\cup\cdots\cup A_j^{(v-1)}\cup A_j^{(v+1)}\cup\cdots\cup A_j^{(d)})\ne\emptyset$, which implies that $A_i^{(v)}\cap A_j^{(q)}\ne\emptyset$ for some $q\in[d]\setminus\{v\}$.
\end{proof}

We shall show that $\mathcal D$ satisfies the desired property in Theorem \ref{weak-sharp}. By the definition of $\mathcal D$, we have
\begin{align*}
|\mathcal{D}_{a_1,\ldots,a_{d-1},k_d}|=&\binom{n-a_d+k_d}{a_1-1,a_2,\ldots,a_{d-1},k_d}+\binom{n-a_d+k_d}{a_1,a_2-1,a_3,\ldots,a_{d-1},k_d}+\cdots\nonumber\\
&+\binom{n-a_d+k_d}{a_1,\ldots,a_{d-2},a_{d-1}-1,k_d},\\
|\mathcal{D}_{a_1,\ldots,a_{d-2},k_{d-1},a_d}|=&\binom{n-a_{d-1}+k_{d-1}}{a_1-1,a_2,\ldots,a_{d-2},k_{d-1},a_d}+\binom{n-a_{d-1}+k_{d-1}}{a_1,a_2-1,a_3,\ldots,a_{d-2},k_{d-1},a_d}+\cdots\nonumber\\
&+\binom{n-a_{d-1}+k_{d-1}}{a_1,\ldots,a_{d-2},k_{d-1},a_d-1},\\
&\cdots,\\
|\mathcal{D}_{k_1,a_2,\ldots,a_d}|=&\binom{n-a_1+k_1}{k_1,a_2-1,a_3,\ldots,a_d}+\binom{n-a_1+k_1}{k_1,a_2,a_3-1,a_4,\ldots,a_d}+\cdots\nonumber\\
&+\binom{n-a_1+k_1}{k_1,a_2,\ldots,a_{d-1},a_d-1}.
\end{align*}
Therefore, it suffices to show that
\begin{align*}
&\sum_{k_1=0}^{a_1-1}\left(\binom{n-a_1+k_1}{k_1,a_2-1,a_3,\ldots,a_d}+\binom{n-a_1+k_1}{k_1,a_2,a_3-1,a_4,\ldots,a_d}+\cdots\nonumber\right.\\
& ~~~~~~ \left.+\binom{n-a_1+k_1}{k_1,a_2,\ldots,a_{d-1},a_d-1}\right) p_1^{k_1}p_2^{a_2}\cdots p_d^{a_d}+\\
&\sum_{k_2=0}^{a_2-1}\left(\binom{n-a_2+k_2}{a_1-1,k_2,a_3,\ldots,a_d}+\binom{n-a_2+k_2}{a_1,k_2,a_3-1,a_4,\ldots,a_d}+\cdots\nonumber\right.\\
& ~~~~~~ \left.+\binom{n-a_2+k_2}{a_1,k_2,a_3,\ldots,a_{d-1},a_d-1}\right)p_1^{a_1}p_2^{k_2}p_3^{a_3}\cdots p_d^{a_d}+\cdots+\\
&\sum_{k_d=0}^{a_d-1}\left(\binom{n-a_d+k_d}{a_1-1,a_2,\ldots,a_{d-1},k_d}+\binom{n-a_d+k_d}{a_1,a_2-1,a_3,\ldots,a_{d-1},k_d}+\cdots\nonumber\right.\\
& ~~~~~~ \left.+\binom{n-a_d+k_d}{a_1,\ldots,a_{d-2},a_{d-1}-1,k_d}\right)p_1^{a_1}\cdots p_{d-1}^{a_{d-1}}p_d^{k_d}=1.
\end{align*}

Suppose that, in the beginning, there are $d$ pockets labeled $1,2,\ldots,d$, each containing a specific number of matches denoted by $a_1,a_2,\ldots ,a_d$, respectively, and in every step we choose one match from the pocket $r$, $r\in [d]$, with probability $p_r$, until $d-1$ of all the $d$ pockets become empty. Let these steps be pairwise independent. If after these steps, all the pockets except for the pocket $1$ become empty, then the last match was taken from one of the $d-1$ pockets $2,3,\ldots,d$, and hence, after these steps, the pocket $1$ contains $a_1-k_1$ (for some $0\le k_1\le a_1-1$) matches (and all of the other pockets are empty) with probability
\begin{align*}
&\left(\binom{n-a_1+k_1}{k_1,a_2-1,a_3,\ldots,a_d}+
\binom{n-a_1+k_1}{k_1,a_2,a_3-1,a_4,\ldots,a_d}+\cdots+\right.\\
&~~~~~~~~~~~~\left.\binom{n-a_1+k_1}{k_1,a_2,\ldots,a_{d-1},a_d-1}\right)p_1^{k_1}p_2^{a_2}\cdots p_d^{a_d}.
\end{align*}
If after these steps, all the pockets except for the pocket $2$ become empty, then the last match was taken from one of the $d-1$ pockets $1,3,\ldots,d$, and hence, after these steps, the pocket $2$ contains $a_2-k_2$ (for some $0\le k_2\le a_2-1$) matches (and all of the other pockets are empty) with probability
\begin{align*}
&\left(\binom{n-a_2+k_2}{a_1-1,k_2,a_3,\ldots,a_d}+\binom{n-a_2+k_2}{a_1,k_2,a_3-1,a_4,\ldots,a_d}+\cdots+\nonumber\right.\\
&\left.\binom{n-a_2+k_2}{a_1,k_2,a_3,\ldots,a_{d-1},a_d-1}\right)p_1^{a_1}p_2^{k_2}p_3^{a_3}\cdots p_d^{a_d}.
\end{align*}
Similarly, we have
\begin{align*}
&\left(\binom{n-a_3+k_3}{a_1-1,a_2,k_3,a_4,\ldots,a_d}+
\binom{n-a_3+k_3}{a_1,a_2-1,k_3,a_4,\ldots,a_d}+\cdots+\right.\\
&\left.\binom{n-a_3+k_3}{a_1,a_2,k_3,a_4,\ldots,a_{d-1},a_d-1}\right)p_1^{a_1}p_2^{a_2}p_3^{k_3}\cdots p_d^{a_d};
\end{align*}
$$\cdots;$$
\begin{align*}
&\left(\binom{n-a_d+k_d}{a_1-1,a_2,\ldots,a_{d-1},k_d}+\binom{n-a_d+k_d}{a_1,a_2-1,a_3,\ldots,a_{d-1},k_d}+\cdots+\nonumber\right.\\
&\left.\binom{n-a_d+k_d}{a_1,\ldots,a_{d-2},a_{d-1}-1,k_d}\right)p_1^{a_1}\cdots p_{d-1}^{a_{d-1}}p_d^{k_d}.
\end{align*}
Obviously, the sum of these probabilities must be equal to $1$.
\end{proof}

\section{Concluding remarks}\label{section-concluding}

This paper investigates five kinds of systems of $d$-partitions of $[n]$, including Symmetric Bollob\'{a}s systems, strong Bollob\'{a}s systems, Bollob\'{a}s systems, skew Bollob\'{a}s systems, and weak Bollob\'{a}s systems. Many known results on variations of Bollob\'{a}s systems are unified. As a corollary of Theorem \ref{main-Bollobas}, Theorem \ref{main-Bollobas-d=3} provides a negative answer to Conjecture \ref{conj}. To give more applications of Theorem \ref{main-Bollobas}, one needs to determine the values of $N_B(d,s)$ for $d\geq 4$.

Even though Conjecture \ref{conj} does not hold for general Bollob\'{a}s systems of $d$-partitions of $[n]$, Theorem \ref{e=1-d-for-symmetric} shows that Conjecture \ref{conj} holds for strong Bollob\'{a}s systems.



\begin{thebibliography}{99}

\bibitem{ap}
P.~Alles and S.~Poljak, Long induced paths and cycles in Kneser graphs, Graphs Comb., 5 (1989), 303--306.

\bibitem{alon}
N.~Alon, An extremal problem for sets with applications to graph theory, J. Comb. Theory Ser. A, 40 (1985), 82--89.

\bibitem{bj}
A.~Bj\"{o}rner, Shellable and Cohen-Macaulay partially ordered sets, Trans. Am. Math. Soc., 260 (1980), 159--183.

\bibitem{B}
B.~Bollob\'{a}s, On generalized graphs, Acta Math. Acad. Sci. Hungar., 16 (1965), 447--452.


\bibitem{frankl}
P.~Frankl, An extremal problem for two families of sets, European J. Comb., 3 (1982), 125--127.

\bibitem{f}
Z.~F\"{u}redi, Geometrical solution of an intersection problem for two hypergraphs, European J. Comb., 5 (1984), 133--136.

\bibitem{hf}
G.~Heged\"{u}s and P.~Frankl, Variations on the Bollob\'{a}s set-pair theorem, European J. Comb., 120 (2024), 103983.

\bibitem{jp}
F.~Jaeger and C.~Payan, Nombre maximal d'ar\^{e}tes d'un hypergraphe $\tau$-critique de rang $h$. C. R. Acad. Sci. Paris S\'{e}r., A-B 273 (1971), A221--A223 (French).

\bibitem{kalai}
G.~Kalai, Intersection patterns of convex sets, Israel J. Math., 48 (1984), 161--174.

\bibitem{kkk}
D.Y.~Kang, J.~Kim, and Y.~Kim, On the Erd\H{o}s-Ko-Rado theorem and the Bollob\'{a}s theorem for $t$-intersecting families, European J. Comb., 47 (2015), 68--74.

\bibitem{katona}
G.O.H.~Katona, Solution of a problem of A. Ehrenfeucht and J. Mycielski, J. Comb. Theory Ser. A, 17 (1974), 265--266.

\bibitem{knpv}
Z.~Kir\'{a}ly, Z.L.~Nagy, D.~P\'{a}lv\"{o}lgyi, and M.~Visontai, On families of weakly cross-intersecting set-pairs, Fund. Inform., 117 (2012), 189--198.

\bibitem{l}
L.~Lov\'{a}sz, Flats in matroids and geometric graphs, in: Combinatorial Surveys, Proc. 6th British Comb. Conf., Academic Press, London, 1977, 45--86.

\bibitem{nv}
J.~O'Neill and J.~Verstraete, A generalization of the Bollob\'{a}s set pairs inequality, Electron. J. Comb., 28 (2021), $\#$P3.8.


\bibitem{sw}
A.~Scott and E.~Wilmer, Combinatorics in the exterior algebra and the Bollob\'{a}s two families theorem, J. Lond. Math. Soc., 104 (2021), 1812--1839.

\bibitem{talbot}
J.~Talbot, A new Bollob\'{a}s-type inequality and applications to $t$-intersecting families of sets, Discrete Math., 285 (2004), 349--353.

\bibitem{t}
T.G.~Tarj\'{a}n, Complexity of lattice-configurations, Studia Sci. Math. Hungar., 10 (1975), 203--211.

\bibitem{t2}
Z.~Tuza, Inequalities for two set systems with prescribed intersections, Graphs Comb., 3 (1987), 75--80.

\bibitem{tuza}
Z.~Tuza, Intersection properties and extremal problems for set systems, in: Irregularities of Partitions (G.~Hal\'{a}sz and V.T.~S\'{o}s, Eds.), Springer-Verlag, 1989, 141--151.

\bibitem{ykxzg}
W.~Yu, X.~Kong, Y.~Xi, X.~Zhang, and G.~Ge, Bollob\'{a}s-type theorems for hemi-bundled two families, European J. Comb., 100 (2022), 103438.

\bibitem{yue}
E.~Yue, Some new Bollob\'{a}s-type inequalities, arXiv:2405.17639.

\end{thebibliography}
\end{document}